\documentclass[12pt]{article}
\usepackage{geometry}
\usepackage{graphicx,tikz,caption,subcaption}
\usepackage{fullpage}
\usepackage{hyperref}
\usepackage{enumerate}
\usepackage{hyperref}
\usepackage{amsmath}
\usepackage{amssymb}
\usepackage{amsfonts}
\usepackage{amsthm}
\usepackage[capitalise]{cleveref}
\usepackage{geometry}
\usepackage{float}
\usepackage{mathrsfs}
\newtheorem{innerproblem}{Problem}
\newenvironment{problem*}
  {
   \innerproblem}
  {\endinnerproblem}
\newtheorem{innerdefinition}{Definition}
\newenvironment{definition*}
  {
   \innerdefinition}
  {\endinnerdefinition}
\newtheorem{theorem}{Theorem}[section]
\newtheorem{lemma}{Lemma}[section]

\newtheorem{claim}{Claim}[section]

\newtheorem{conjecture}{Conjecture}[section]
\newtheorem{observation}{Observation}[section]
\newtheorem{observation*}{Observation}[section]

\newtheorem{problem}{Problem}[section]

\newtheorem{case}{Case}

\newtheorem{subcase}{Subcase}[case]

\newcommand{\sat}{\mathrm{sat}}

\newcommand{\SAT}{\mathrm{SAT}}

\usetikzlibrary{patterns}
\usetikzlibrary{shapes}
\usetikzlibrary{decorations.pathreplacing}
\usetikzlibrary{decorations.pathmorphing}
\usetikzlibrary{positioning}
\definecolor{mypink}{RGB}{255, 105, 160}
\definecolor{myorange}{RGB}{255, 178, 49}
\definecolor{darktangerine}{rgb}{1.0, 0.66, 0.07}
\definecolor{darkpastelgreen}{rgb}{0.01, 0.75, 0.24}
\title{Saturation numbers of some joins of graphs}
\author{Xinying Hua \,  Yuejian Peng\thanks{E-mail addresses:  xyhuamath@163.com (X. Hua), ypeng1@hnu.edu.cn (Y. Peng, corresponding author), supported in part by National Natural Science Foundation of China (No. 12571363) and   Natural Science Foundation of Hunan Province (Grant No. 2025JJ30003).}\\
{\footnotesize  School of Mathematics, Hunan University, Changsha, Hunan, 410082, P.R. China}\\
}

\begin{document}

\maketitle
\begin{abstract}
Let $H$ be a graph. A graph $G$ is $H$-saturated if $G$ is $H$-free, but adding any edge between two non-adjacent vertices of $G$ yields an $H$-copy as a subgraph. The saturation number $\sat(n, H)$ is the minimum number of edges in an $H$-saturated graph on $n$ vertices.
The saturation number for the join of a vertex and a graph $F$, denoted by $K_1\vee F$, has attracted considerable attention.
Cameron and Puleo [Discrete Math. 345 (2022), 112867] proved that $\sat(n,K_1 \vee F)\le n-1+\sat(n-1, F)$ for $n > |V(F)|$.
A natural question is when the above equality holds. Most existing results impose conditions on $F$ and assume that $F$ has no isolated vertices.
Let $K_p^-$ be the graph obtained by deleting one edge from the complete graph $K_p$.
In this paper, we investigate the saturation number of $K_1\vee F$ when $F$ contains isolated vertices, and determine the exact value of $\sat(n, K_1\vee F)$ when $F=K^-_{3}\cup sK_1(s\ge 1)$ or $F=K^-_{p-1}\cup K_1(p\ge 5)$. In our results,  $\sat(n,K_1 \vee F)= n-1+\sat(n-1, F)$ holds when $F=K^-_{3}\cup sK_1$ for any $s\ge 1$, but fails when $F=K^-_{p-1}\cup K_1$ for $p\ge 5$.
\end{abstract}
\section{Introduction}

All graphs considered in this paper are simple and undirected.  Let $G=(V(G),E(G))$ be a graph.
We use $|V(G)|$ to denote its order and $e(G) = |E(G)|$ to denote its number of edges. For a fixed graph $F$, a graph $G$ is called \emph{$F$-saturated} if it contains no subgraph isomorphic to $F$, while for every edge $e\notin E(G)$, adding $e$ to $G$ produces a copy of $F$ as a subgraph. The \emph{saturation number} of $F$, denoted by $\sat(n,F)$, is the minimum number of edges among all $F$-saturated graphs of order $n$. Let $$\SAT(n,\,F):= \{G : |V(G)|=n,  e(G) = \sat(n, F), \mbox{ and } G \mbox{ is } F\mbox{-saturated}\},$$ and each graph in $\SAT(n,\,F)$ is called an \emph{extremal graph} for $F$.
For any two vertex-disjoint graphs $G_1$ and $G_2$, their \emph{join graph} $G_1\vee G_2$ is defined as the graph formed from the disjoint union $G_1\cup G_2$ by adding all edges between $V(G_1)$ and $V(G_2)$.
We denote the path, cycle, star, complete graph, and empty graph on $n$ vertices by $P_n$, $C_n$, $S_n$, $K_n$, and $\overline{K_n}$, respectively. We also let $K_p^-$ denote the graph obtained by removing one edge from the complete graph $K_p$.

The study of the saturation number of a graph was initiated by Erd\H{o}s, Hajnal and Moon. In~\cite{erdos1964}, they proved that $\sat(n, K_{k+1})=(k-1)n-\binom{k}{2}$, with
$K_{k-1}\vee \overline{K_{n-k+1}}$ being the unique minimum-size $K_{k+1}$-saturated graph of order $n$.
K\'{a}szonyi and Tuza~\cite{Kaszonyi} gave a general upper bound on $\sat(n, H)$ for any graph $H$,
and determined $\sat(n, P_k)$, $\sat(n, K_{1,k})$, and $\sat(n, kP_2)$.
In subsequent work, the saturation number has been investigated for a wide variety of graph families.
Examples include small cycles~\cite{Chen2009,Chen2011,Lan2025,Ma2021,Ollmann1972,Tuza1989}, complete multipartite graphs \cite{Chen2014,He2021,Huang2024}, forests~\cite{Cao2023,Chen2015,Fan2015,Faudree2009,He2023,Kaszonyi,Lv2023}, books~\cite{GChen2009}, and disjoint unions of complete graphs~\cite{Chen2024,Faudree2009b,Zhu2025}. Nevertheless, results providing the exact saturation number for all $n$ are still quite limited.
See \cite{Currie2021} and \cite{Faudree2011} for comprehensive surveys.

The join operation is a fundamental and widely used construction in graph theory.
Existing work on the saturation number for join graphs has focused on determining $\sat(n, K_1 \vee F)$ from the saturation number of $F$.
Cameron and Puleo \cite{Ca} showed that $\sat(n,K_1 \vee F)\le n-1+\sat(n-1, F)$ for $n \ge |V(F)|+1$. It is an interesting problem to determine for which graphs the above inequality is tight.
In 2024, Hu, Luo and Peng \cite{HU} proved the following theorem.

\begin{theorem}\label{T12} \text{\bf \cite{HU}} Let $s$ and $n$ be positive integers, and let $F$ be a graph without isolated vertices. Then for $n \geq 3s^2 -s +2\,\sat(n-s, F)+1$, we have
\begin{equation*}
\begin{aligned}
\sat(n, K_s \vee F) = \binom{s}{2} + s(n-s) + \sat(n-s, F) \,.
\end{aligned}
\end{equation*}
\end{theorem}
In~\cite{Qiu}, Qiu, He, Lu and Xu noted that the condition on $F$ in Theorem~\ref{T12} implies that $F$ must contain an isolated edge. Subsequently,
Hu, Ji and Cui~\cite{HU1} proved that $\sat(n,K_1 \vee P_t)=n-1+\sat(n-1,P_t)$ holds for $t \geq 5$ and sufficiently large $n$.
Song, Hu, Ji and Cui \cite{Song} proved that $\sat(n,K_1 \vee C_4)=\lfloor\frac{5n-10}{2}\rfloor$ holds for all $n \geq 6$.
In~\cite{Qiu}, Qiu, He, Lu and Xu further showed that $\sat(n,K_1 \vee C_t)=n-1+\sat(n-1, C_t)$ holds if $t \geq 8$ and $n \geq 56t^3$.

Among these results, almost all graphs $F$ under consideration are connected graphs or disconnected graphs without isolated vertices. Motivated by this, Hua and Peng~\cite{Hua2} considered the case in which $F$ consists of complete graphs and some isolated vertices and obtained the exact value of $\sat(n, K_1\vee (K_{p-1}\cup K_1))$ for all $p\ge 4$. Moreover, when $p=3$, the authors extended the result to any number of isolated vertices and determined the saturation number for $K_1\vee F$ when $F=K_{2}\cup qK_1$, or $F=2K_{2}\cup qK_1$ for any $q\ge 1$. In addition, all extremal graphs in \cite{Hua2} are fully characterized.

In this paper, we investigate the saturation number of $K_1 \vee F$ when $F$ is the vertex-disjoint union of the nearly complete graph $K_p^-$ and some isolated vertices.
Let $$K_{p^-}^{+s}:=K_1\vee(K^-_{p-1}\cup sK_1).$$
Our results are as follows.
\begin{theorem}\label{thm:K4*}
For $n\ge  s+4$, $\sat(n, K_{4^-}^{+s})=\lceil\frac{3n-4}{2}\rceil$.
\end{theorem}
For $s=1$, we obtain the saturation number of $K_{p^-}^{+1}$ for every $p\ge 5$.
\begin{theorem}\label{thm:Kp*}
Let $p\ge 5$, $n\ge p+1$ and $n\equiv k\pmod{p}$, where $k\in [0, p-1]$. Then $\sat(n, K_{p^-}^{+1})=\frac{(p-1)n}{2}+\frac{k(k-p)}{2}$.
\end{theorem}

The rest of this paper is organized as follows. In Section 2, we introduce some necessary definitions and notations. In Section 3, we prove Theorem~\ref{thm:K4*}. In Section 4, we prove Theorem~\ref{thm:Kp*}.
We conclude with remarks in the final section.

\section{Definitions and Notation}

Given two disjoint sets $A, B \subseteq V(G)$, let $e(A, B)$ be the number of edges in $G$ with one endpoint in $A$ and the other in $B$. Given two vertices $u,v\in V(G)$, the \emph{distance} between $u$ and $v$, denoted by $d(u,v)$, is the length of a shortest $(u,v)$-path in $G$. We write $[n]=\{1,2,\cdots,n\}$.

For each vertex $v \in V(G)$, let $N_G(v)$ be the \emph{open neighborhood} of $v$,
and $N_G[v] = N_G(v) \cup\{v\}$  be its \emph{closed neighborhood}. The \emph{degree} of $v$ is given by $d_G(v) := |N_G(v)|$, and the minimum degree of $G$ is defined as $\delta(G) := \min\limits_{v\in V(G)} d_G(v)$.
For a vertex subset $D \subseteq V(G)$, we use the notation $N_D(v) = N_G(v) \cap D$, $N_G(D) = \bigcup_{v\in D} N_G(v)$, and $N_G[D] = D \cup N_G(D)$.
For $U\subseteq V(G)$, let $G-U$ be the graph obtained by deleting the vertex set $U$ and all edges incident to $U$. Also, we let $G[U]$ be the induced subgraph of $G$ on $U$.

 A \emph{cut vertex} of a graph is a vertex whose deletion increases the number of components.
A vertex of degree one is said to be a \emph{leaf}, and the edge incident to a leaf is called a \emph{pendant edge}.

\section{Saturation number for $K_{4^-}^{+s}$}
In this section, we consider the saturation number of $K_{4^-}^{+s}$.
Recall the statement of Theorem~\ref{thm:K4*}.

\noindent\textbf{Theorem~\ref{thm:K4*}.}
For $n\ge  s+4$, $\sat(n, K_{4^-}^{+s})=\lceil\frac{3n-4}{2}\rceil$.

\begin{proof}[\textbf{Proof of Theorem~\ref{thm:K4*}}]

If $n$ is odd, $K_1\vee\frac{n-1}{2}K_2$ is a $K_{4^-}^{+s}$-saturated graph with $\lceil\frac{3n-4}{2}\rceil$ edges.
If $n$ is even, then $K_1\vee(K_1\cup \frac{n-2}{2}K_2)$ is a $K_{4^-}^{+s}$-saturated graph with $\lceil\frac{3n-4}{2}\rceil$ edges.
Thus, $\sat(n, K_{4^-}^{+s})\le \lceil\frac{3n-4}{2}\rceil$.
Now, we prove the lower bound.

Let $G$ be a minimum $K_{4^-}^{+s}$-saturated graph of order $n$.

\setcounter{claim}{0}

\begin{claim} \label{clm:degreesum}
If $\sum\limits_{u\in V(G)}d_G(u)\ge 3n-4$, $e(G)\ge \lceil\frac{3n-4}{2}\rceil$.
\end{claim}
\begin{proof}
Since $2e(G)=\sum\limits_{u\in V(G)}d_G(u)$, if $\sum\limits_{u\in V(G)}d_G(u)= 3n-4$, then $3n-4$ is even and $e(G)=\frac{3n-4}{2}=\lceil\frac{3n-4}{2}\rceil$. If $\sum\limits_{u\in V(G)}d_G(u)\ge 3n-3$, then $e(G)\ge \lceil\frac{3n-4}{2}\rceil$. This completes the proof.
\end{proof}

\setcounter{case}{0}
	We divide the proof into the following two cases.
\begin{case}
$G$ is connected.
\end{case}

Let $A_i=\{x\in V(G)|d_G(x)=i\}$ and $a_i=|A_i|$. If $A_1\cup A_2=\emptyset$,
then $e(G)\ge\frac{3n}{2}>\lceil\frac{3n-4}{2}\rceil$. So we assume that $A_1\cup A_2\neq\emptyset$.

\begin{claim}\label{clm:distance}
If $x$ and $y$ are two vertices in $A_1\cup A_2$, then $d(x,y)\le 2$.
\end{claim}

\begin{proof}
If $d(x,y)\ge 3$, then $G+xy$ contains a $K_{4^-}^{+s}$-copy containing $xy$, and $xy$ must be the pendant edge of this copy, say $H$. Then $d_H(x)=s+3$ or $d_H(y)=s+3$.
Furthermore, $d_G(x)=s+2$ or $d_G(y)=s+2$.
As $s\ge 1$, $d_{G}(x)\ge  3$ or $d_{G}(y)\ge  3$, a contradiction.
\end{proof}

\begin{claim} \label{clm:oneleaf}
$a_1\le 1$.
\end{claim}
\begin{proof}
Otherwise, assume that $x$ and $y$ are two leaves in $G$.
Since $n\ge s+4$ and $G$ is connected, $xy\notin E(G)$.
Then $G+xy$ has a $K_{4^-}^{+s}$-copy containing $xy$.
Since each edge in $K_{4^-}^{+s}$ is incident with a vertex of degree at least 3, and $d_{G+xy}(x)= d_{G+xy}(y)=2$,
$xy$ cannot be any edge in a $K_{4^-}^{+s}$-copy, a contradiction.
\end{proof}

Thus, we consider the following two cases.

\begin{subcase}
$a_1=1$.
\end{subcase}
If $a_2\le 2$, then $\sum_{u\in V(G)}d_G(u)\ge 1+2a_2+3(n-1-a_2)=3n-2-a_2\ge 3n-4$.
By Claim~\ref{clm:degreesum}, $e(G)\ge \lceil\frac{3n-4}{2}\rceil$.
If $a_2\ge 3$, let $v'$ be the unique leaf in $G$ and $v$ be the neighbour of $v'$.
By Claim~\ref{clm:distance}, $d(v', x)\le 2$ for every $x\in A_2$.
Thus, each vertex in $A_2$ is adjacent to $v$.
Thus, $d_{G}(v)\geq a_2+1$
and $$\sum\limits_{u\in V(G)}d_G(u)\ge d_G(v)+1+2a_2+3(n-2-a_2)\ge (a_2+1)+1+2a_2+3(n-2-a_2)= 3n-4.$$
 By Claim~\ref{clm:degreesum}, $e(G)\ge \lceil\frac{3n-4}{2}\rceil$.

\begin{subcase}
$a_1=0$.
\end{subcase}
First we assume that $a_2\le 4$. Then $\sum\limits_{u\in V(G)}d_G(u)\ge 2a_2+3(n-a_2)=3n-a_2\ge 3n-4$.
Together with Claim~\ref{clm:degreesum}, $e(G)\ge \lceil\frac{3n-4}{2}\rceil$.

Next we consider the case when $a_2\ge 5$.
If $\cap_{u\in A_2} N_G(u)\neq\emptyset$, then
$G$ contains a vertex, say $x$, such that $A_2\subseteq N_G(x)$. So, $d_G(x)\ge a_2$ and $\sum\limits_{u\in V(G)}d_G(u)\ge d_G(x)+2a_2+3(n-1-a_2)\ge 3n-3$.
Thus, $e(G)\ge \lceil\frac{3n-4}{2}\rceil$.
Now we assume that $\cap_{u\in A_2} N_G(u)=\emptyset$.

If there exist two vertices $x$ and $y$ in $A_2$ such that $xy\in E(G)$, then
$N_G(x)\cap N_G(y)= \emptyset$.
Otherwise, assume that $N_G(x)\cap N_G(y)=\{z\}$. Then for any $v\in A_2\backslash \{x, y\}$, since $d(x, v)\le 2$ and $d(y, v)\le 2$ by Claim~\ref{clm:distance}, $vz\in E(G)$.
It follows that $\cap_{u\in A_2} N_G(u)\neq\emptyset$, a contradiction.
Let $N_G(x)\backslash\{y\}=\{x_1\}$ and $N_G(y)\backslash\{x\}=\{y_1\}$, see Fig.~\ref{fig:K_4*} (a).
Since for any $z\in A_2\backslash\{x, y\}$, $d(x,z)\le 2$ and $d(y,z)\le 2$, we have $N_G(z)=\{x_1, y_1\}$.
Thus, $d_G(x_1)+d_G(y_1)\ge 2(a_2-2)+1+1=2a_2-2$.
Hence,
\begin{align*}
\sum_{u\in V(G)}d_G(u) &\ge d_G(x_1)+d_G(y_1)+2a_2+3(n-2-a_2) \\
& \ge (2a_2-2)+2a_2+(3n-6-3a_2)\\
& =3n+a_2-8\\
& \ge 3n-3,
\end{align*}
 which gives $e(G)\ge \lceil\frac{3n-4}{2}\rceil$.
 \begin{figure}[H]
\begin{center}
\includegraphics*[width=14cm]{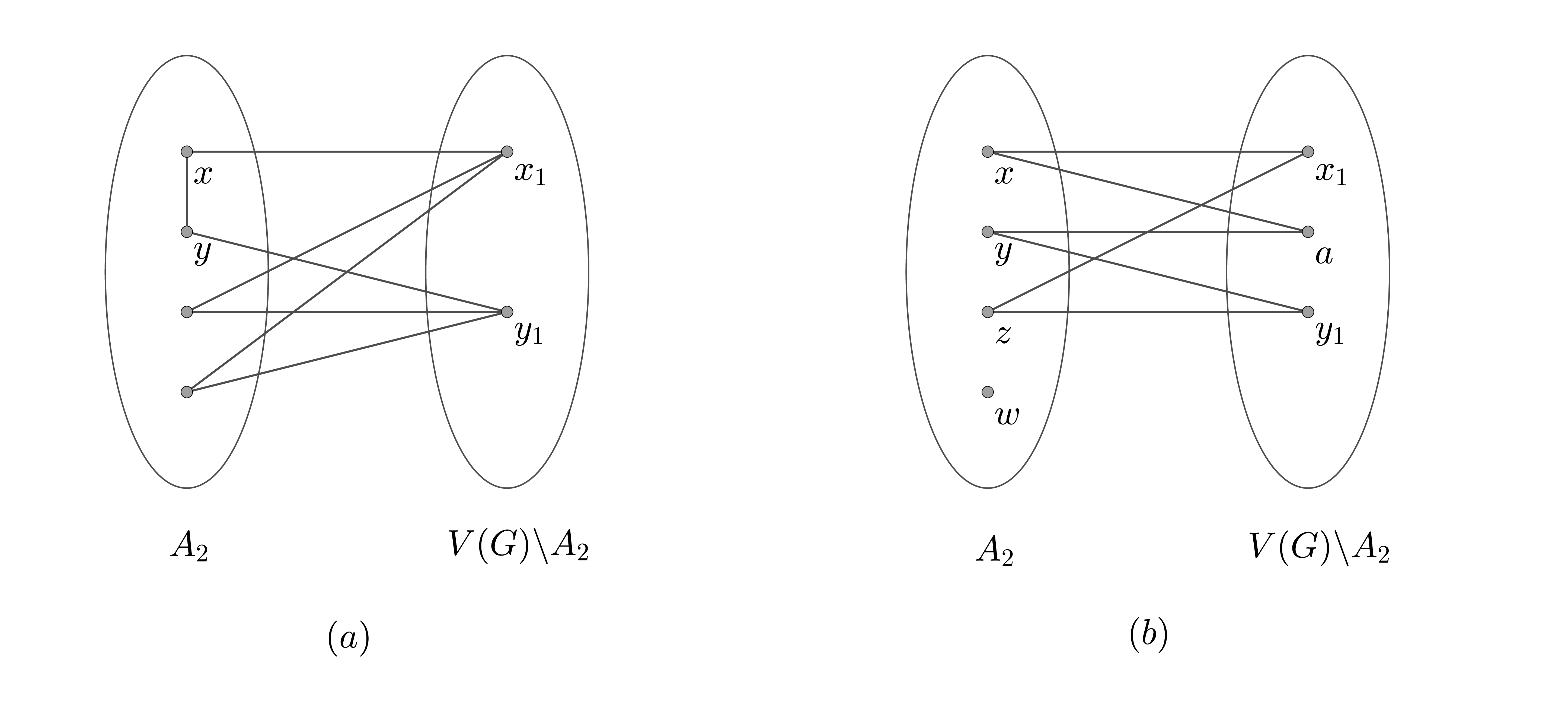}
\end{center}
\vspace*{-0.5cm}
\caption{Two graphs occurred in the proof of Subcase 1.2.}
\label{fig:K_4*}
\end{figure}

Now, we assume that $G$ has no adjacent vertices of degree 2.
This means for any $u \in A_2$, each neighbour of $u$ is of degree at least 3.
Also, by Claim~\ref{clm:distance}, any two vertices in $A_2$ share a common neighbour.
Combining this fact with $\cap_{u\in A_2} N_G(u)=\emptyset$, there exist two vertices $x,y\in A_2$ such that $x$ and $y$ have a unique common neighbour, say $a$.
Let $N_G(x)\backslash\{a\}=\{x_1\}$ and $N_G(y)\backslash\{a\}=\{y_1\}$, see Fig.~\ref{fig:K_4*} (b).
Since there exists a vertex $z\in A_2$ such that $za\notin E(G)$, we have $zx_1\in E(G)$ and $zy_1\in E(G)$.
Then for any $w\in A_2\backslash\{x,y,z\}$, $|N_G(w)\cap \{a, x_1, y_1\}|=2$.
Otherwise, $d(w, x)>2$ or $d(w, y)>2$ or $d(w, z)>2$, a contradiction.
So, $d_G(a)+d_G(x_1)+d_G(y_1)\ge 2a_2$. By our assumption that $a_2\ge 5$,
we have$$\sum\limits_{u\in V(G)}d_G(u)\ge d_G(a)+d_G(x_1)+d_G(y_1)+2a_2+3(n-3-a_2)\ge 3n+a_2-9\ge 3n-4.$$
Thus, by Claim~\ref{clm:degreesum}, $e(G)\ge \lceil\frac{3n-4}{2}\rceil$.

Combining the above arguments, $e(G)\ge \lceil\frac{3n-4}{2}\rceil$ if $G$ is connected.

\begin{case}
$G$ is disconnected.
\end{case}
Suppose that $G_{1}, G_{2},\ldots, G_{t}\,(t\geq 2)$ are the components of $G$. If $\delta(G)\ge 3$, then $e(G)> \lceil\frac{3n-4}{2}\rceil$. So, we may assume that there exists a vertex $u\in V(G_1)$ with $d_{G}(u)\le 2$. Take $v\in V(G)\backslash V(G_1)$. Then $G+uv$ has a $K_{4^-}^{+s}$-copy, say $H$, containing $uv$. Since $uv$ is not contained in any triangle of $G+uv$, $uv$ must be the pendant edge in $H$. As $d_{G+uv}(u)\le 3$, we have $d_H(u)=1$ and $d_H(v)=s+3$. Thus, every vertex $v\in V(G)\backslash V(G_1)$ satisfies $d_G(v)=s+2\ge 3$.
Let $|V(G_1)|=n_1$. If $n_1\le s+3$, then $G_1 \cong K_{n_1}$.
Since $n\ge 5$, $$\sum\limits_{u\in V(G)}d_G(u)\ge3(n-n_1)+n_1(n_1-1)=3n+(n_1-2)^2-4\ge 3n-4.$$
By Claim~\ref{clm:degreesum}, $e(G)\ge \lceil\frac{3n-4}{2}\rceil$.
If $n_1\ge s+4$, then $G_1$ is a connected $K_{4^-}^{+s}$-saturated graph.
From the connected case, we have $\sum\limits_{u\in V(G_1)}d_G(u)\ge 3n_1-4$. Thus,
$$\sum\limits_{u\in V(G)}d_G(u)=\sum\limits_{u\in V(G_1)}d_G(u)+\sum\limits_{u\in V(G)\backslash V(G_1)}d_G(u)\ge (3n_1-4)+3(n-n_1)= 3n-4.$$
By Claim~\ref{clm:degreesum}, $e(G)\ge \lceil\frac{3n-4}{2}\rceil$.

This completes the proof.
\end{proof}

\section{Saturation number for $K_{p^-}^{+1}\,(p\ge 5)$}

In this section, we prove Theorem~\ref{thm:Kp*}. We recall the statement of Theorem~\ref{thm:Kp*}.

\noindent\textbf{Theorem~\ref{thm:Kp*}.}
Let $p\ge 5$, $n\ge p+1$ and $n\equiv k\pmod{p}$. Then $$\sat(n, K_{p^-}^{+1})=\frac{(p-1)n}{2}+\frac{k(k-p)}{2}.$$

\subsection{Preliminary results}
We first prove several results that are used repeatedly in the proof of Theorem~\ref{thm:Kp*}.

\begin{observation}\label{ob:no K_p}
Let $G$ be a $K_{p^-}^{+1}$-saturated graph and $G'$ be any component of $G$.
If $|V(G')|\ge p+1$, then $G'$ contains no $K_p$-copy.
In particular, if $G$ is connected and $|V(G)|\ge p+1$, then $G$ contains no $K_p$-copy.
\end{observation}

\begin{lemma}\label{lem:N has K_p}
Let $G$ be a connected $K_{p^-}^{+1}$-saturated graph with $|V(G)|\ge p+1$. If $u$ is a vertex in $G$ such that $G[N_G[u]]$ contains a $K_{p}^-$-copy, then $d_G(u)=p-1$ and $G[N_G[u]]\cong K_{p}^-$.
\end{lemma}
\begin{proof}
Assume that $M'$ is a $K_{p}^-$-copy in $G[N_G[u]]$. Clearly, $u\in V(M')$, for otherwise, $G[V(M')\cup\{u\}]$ contains a $K_{p^-}^{+1}$-copy, a contradiction.
If $d_{M'}(u)=p-2$, since $M'$ is a subgraph of $G[N_G[u]]$, $G[N_G[u]]$ contains a $K_p$-copy, a contradiction by Observation~\ref{ob:no K_p}. So, $d_{M'}(u)=p-1$.
If $d_G(u)\ge p$, there exists a vertex $w\in N_G(u)\backslash V(M')$ such that $M'+uw$ is a $K_{p^-}^{+1}$-copy in $G$, a contradiction. So, $d_G(u)=d_{M'}(u)=p-1$, and thus, $N_G[u]=V(M')$.
It follows that $G[N_G[u]]\cong K_p^-$ or $K_p$. Since $G$ is $K_p$-free by Observation~\ref{ob:no K_p}, we have $G[N_G[u]]\cong K_p^-$.
\end{proof}

\begin{lemma}\label{lem:type2}
Let $G$ be a graph and $M'$ be a copy of $K_p^-$ in $G$. If $u,v\in V(M')$, then $|N_G(u)\cap N_G(v)|\ge p-3$ and $G[N_G(u)\cap N_G(v)]$ contains a $K_{p-3}$-copy.
\end{lemma}

\begin{proof}
Suppose that $uv\in E(G)$ and, contrary to the claim, that $G[N_G(u)\cap N_G(v)]$ contains no $K_{p-3}$-copy. Then $uv$ does not belong to any $K_{p-1}$-copy, and thus, $uv$ cannot be an edge of a $K_p^-$-copy, a contradiction. So, $u$ and $v$ are the unique pair of non-adjacent vertices in $M'$. This directly implies that $|N_G(u)\cap N_G(v)|\ge p-3$ and $G[N_G(u)\cap N_G(v)]$ contains a $K_{p-3}$-copy.
\end{proof}

\begin{lemma}\label{lem:type2.2}
Let $v$ be a vertex in a graph $G$ such that $G[N_G(v)]$ contains no $K_{p-2}$-copy. Then for each $v'\in N_G(v)$, we have $G[N_G[v']]\ncong K_p^-$.
\end{lemma}
\begin{proof}
Assume that $G[N_G[v']]\cong K_p^-$, say $M'$. Since $vv'\in E(G)$, $v\in V(M')$.
Since $G[N_G(v)]$ contains no $K_{p-2}$-copy, $v$ does not belong to any $K_{p-1}$-copy.
So, $v$ cannot be a vertex of any $K_p^-$-copy, a contradiction.
\end{proof}
\begin{lemma}\label{lem:x in two K}
Let $G$ be a connected $K_{p^-}^{+1}$-saturated graph with $|V(G)|\ge p+1$, and let $M_1, M_2$ be two distinct $K_p^-$-copies in $G$. If $x\in V(M_1)\cap V(M_2)$, then $d_{M_1}(x)=d_{M_2}(x)=p-2$.
\end{lemma}

\begin{proof}
Otherwise, assume that $d_{M_1}(x)=p-1$. Then $M_1$ is a $K_p^-$-copy in $G[N_G[x]]$.
By Lemma~\ref{lem:N has K_p}, $d_G(x)=p-1$ and $G[N_G[x]]\cong K_p^-$. So, $G[N_G[x]]=M_1$.
If $d_{M_2}(x)=p-1$, then $d_{M_2}(x)=d_G(x)$ and $V(M_2)=N_G[x]=V(M_1)$. Note that $G[N_G[x]]\cong K_p^-$. Then $M_1=M_2$, a contradiction.
So, $d_{M_2}(x)=p-2$.

Let $V(M_1)=\{x,x_1,\ldots, x_{p-1}\}$ and $d_{M_1}(x_{p-2})=d_{M_1}(x_{p-1})=p-2$, see Figure~\ref{fig:K_p^*}.
Since $d_{M_2}(x)=p-2$, the vertices of $N_G[x]$ induce a $K_{p-1}$-copy in $M_2$. As $N_G[x]\subseteq V(M_1)$, we may assume that $G[\{x,x_1,\ldots, x_{p-2}\}]$ is exactly this $K_{p-1}$.
Let $V(M_2)\backslash N_G[x]=\{x^*\}$. Then $x^*$ is adjacent to all vertices in $M_2$, except for $x$.
Thus, $x_1x^*\in E(G)$ and $M_1+x_1x^*$ is a $K_{p^-}^{+1}$-copy in $G$, a contradiction.

Summarizing the above, $d_{M_1}(x)=d_{M_2}(x)=p-2$.
\end{proof}
 \begin{figure}[H]
\begin{center}
\includegraphics*[width=11cm]{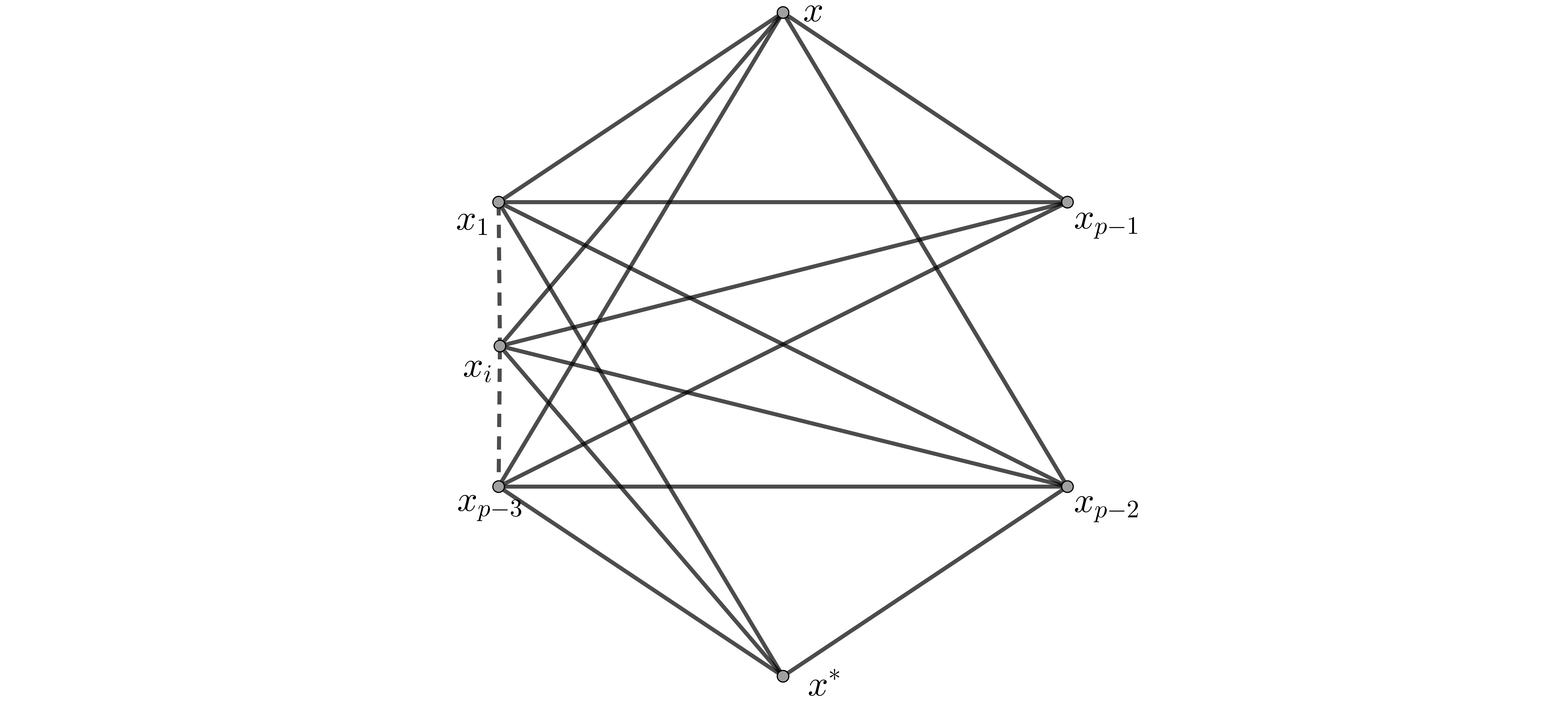}
\end{center}
\vspace*{-0.7cm}
\caption{The local structure of $G$ in Lemma~\ref{lem:x in two K}.}
\label{fig:K_p^*}
\end{figure}

\begin{lemma}\label{kpsubgraph}
Let $G$ be a disconnected $K_{p^-}^{+1}$-saturated graph of order $n\ge p+1$. Then all but at most one component of $G$ are $K_p$-copies.
\end{lemma}
\begin{proof}
Assume that $G_1, G_2,\ldots, G_t\,(t\ge 2)$ are the components of $G$ and that there exist $1\le i<j\le t$ such that $G_i$ and $G_j$ are both $K_p$-free. We claim that
if $G_i\, (i\in [t])$ is $K_p$-free, then $G_i$ contains a vertex $u$ such that $G[N_G[u]]$ contains no $K_p^-$-copy.

Suppose on the contrary that for any $u\in V(G_i)$, $G[N_G[u]]$ contains a $K_p^-$-copy.
Since $G_i$ is $K_p$-free, it follows that $G_i$ must be a connected $K_{p^-}^{+1}$-saturated graph with $|V(G_i)|\ge p+1$. By Lemma~\ref{lem:N has K_p}, $G[N_G[u]]\cong K_p^-$, say $M_1$.
Let $V(M_1)=\{u,u_1,\ldots, u_{p-1}\}$ and $d_{M_1}(u_{p-2})=d_{M_1}(u_{p-1})=p-2$.
Since $u_{p-2}\in V(G_i)$, by the assumption, we also have $G[N_G[u_{p-2}]]\cong K_p^-$, say $M_2$. Clearly, $d_{M_2}(u_{p-2})=p-1$.
Then $u_{p-2}\in V(M_1)\cap V(M_2)$, but $d_{M_2}(u_{p-2})=p-1$, a contradiction by Lemma~\ref{lem:x in two K}.

Thus, we can take $u\in V(G_i)$ and $v\in V(G_j)$ such that $G[N_G[u]]$ and $G[N_G[v]]$ are both $K_p^-$-free.
By the definition of $K_{p^-}^{+1}$-saturated graph, $G+uv$ has a $K_{p^-}^{+1}$-copy $M$ containing $uv$.
Since $uv$ is not contained in any $K_p^-$-copy of $G+uv$, $uv$ must be the pendant edge of $M$. Thus, $G[N_G[u]]$ or $G[N_G[v]]$ contains a $K_p^-$-copy, a contradiction.

So, there is at most one component of $G$ that is $K_p$-free.
Moreover, each component containing a $K_p$ has exactly $p$ vertices, for otherwise $G$ contains a $K_{p^-}^{+1}$-copy, a contradiction.
Therefore, all but at most one component of $G$ are $K_p$-copies.
\end{proof}

\subsection{Proof of Theorem 1.3}
Now we prove the saturation number of $K_{p^-}^{+1}$ when $p\ge 5$. Let $n\equiv k\pmod{p}$ and $$h(n, p)=\frac{(p-1)n}{2}+\frac{k(k-p)}{2}.$$ For $x\in V(G)$ and $S\subseteq V(G)$, let $d_S(x)=|N_G(x)\cap S|$ and $\overline{S} = V(G) \setminus S$.

\begin{lemma}\label{lem:connected}
Let $G$ be a connected $K_{p^-}^{+1}$-saturated graph of order $n\ge p+1$. Then $e(G)\ge h(n, p)$.
\end{lemma}

\begin{proof} If $\delta(G)\ge p-1$,
then $e(G)\ge\frac{(p-1)n}{2}\ge h(n, p)$. Now, suppose that $\delta(G)\le p-2$.
Let $v$ be a vertex of minimum degree in $G$. Assume that $d_G(v)=s(\le p-2)$ and $N_G(v)=\{v_1,v_2,\ldots, v_s\}$. Let $V_1=N_G(v)$ and $V_2=V(G)\backslash (V_1\cup \{v\})$.

Take $u\in V_2$. Then $G+uv$ contains a $K_{p^-}^{+1}$-copy $M_u$ with $uv\in E(M_u)$.
Note that within a $K_{p^-}^{+1}$-copy, the edge $uv$ has one of two possibilities: either $uv$ is the pendant edge, or $uv$ is contained in the unique $K_p^-$-copy of $K_{p^-}^{+1}$.
Based on this, we can partition $V_2$ into two subsets: $A=\{u\in V_2: uv \mbox{ is the pendant edge in } M_u\}$ and $B=\{u\in V_2: uv \mbox{ is an edge in the } K_p^- \mbox{-copy of } M_u\}$.
\setcounter{claim}{0}
\begin{claim}\label{clm:type A}
For any $u\in A$, $d_G(u)=p-1$ and $G[N_G[u]]\cong K_p^-$.
\end{claim}

\begin{proof}
 Since $uv$ is the pendant edge of $M_u$ in $G+uv$, and $d_{G+uv}(v)\le p-1$, we have $d_{M_u}(u)=p$, Then $G[N_G[u]]$ contains a $K_p^-$-copy. By Lemma~\ref{lem:N has K_p}, $d_G(u)=p-1$ and $G[N_G[u]]\cong K_p^-$.
\end{proof}

\begin{claim}\label{clm:AB disjoint}
$A\cap B=\emptyset$.
\end{claim}
\begin{proof}
Assume that $x\in A\cap B$. Then $G+vx$ contains a $K_{p^-}^{+1}$-copy $M_x$ such that $vx\in E(M_x)$.
 Since $x\in B$, by Lemma~\ref{lem:type2}, $G[V_1\cap N_G(x)]$ contains a $K_{p-3}$-copy. Together with $p\ge 5$, $G[V_1\cap N_G(x)]$ contains an edge.

Since $x\in A$, by Claim~\ref{clm:type A}, $G[N_G[x]]\cong K_p^-$, say $M'$.
Together with the fact that $G[V_1\cap N_G(x)]$ contains an edge, $M'$ has an edge in $G[V_1]$.
Thus, $V_1$ contains a $(p-1)$-degree vertex of $M'$, say $y$.
Note that $v\notin N_G(x)$. Then $v\notin V(M')$ and $M'+vy$ is a $K_{p^-}^{+1}$-copy in $G$, a contradiction.
This completes the proof of Claim~\ref{clm:AB disjoint}.
\end{proof}
Let $a=|A|$ and $b=|B|$. By Claim~\ref{clm:AB disjoint}, $a+b=n-s-1$. By Claim~\ref{clm:type A}, for any $u\in A$, $G[N_G[u]]\cong K_p^-$; we denote this $K_p^-$-copy by $M_u$. Also, we denote by $u_1$ and $u_2$ the two vertices of degree $p-2$ in $M_u$.
Then we have the following claim.

\begin{claim}\label{clm:A}
For any $u\in A$, $N_{\overline{A}}(u)=\{u_1, u_2\}$. Consequently, if $A\neq \emptyset$, then $a\ge p-2$, and there exist two vertices in $\overline{A}$ each of which has $p-2$ neighbours in $A$.
\end{claim}
\begin{proof}
Since $N_G[u]=V(M_u)$, it suffices to prove that $V(M_u)\cap A=V(M_u)\backslash\{u_1, u_2\}$. Clearly, $v\notin N_G(u)$. Then $V(M_u)\subseteq V_1\cup V_2$.

Let $x$ be a vertex in $M_u$ such that $d_{M_u}(x)=p-1$. If $x\in V_1$, since $v\notin V(M_u)$, $M_u+vx$ is a $K_{p^-}^{+1}$-copy in $G$, a contradiction. So, $x\in V_2$.
Notice that $M_u$ is also a $K_p^-$-copy in $G[N_G[x]]$. Then $G[N_G[x]]\cong K_p^-$ by Lemma~\ref{lem:N has K_p}. Furthermore, $vx$ is the pendant edge of a $K_{p^-}^{+1}$-copy in $G+vx$.
Thus, $x\in A$.

Now, we claim that $A\cap \{u_1, u_2\}=\emptyset$. Otherwise, we can assume that $u_1\in A$.
By Claim~\ref{clm:type A}, $G[N_G[u_1]]\cong K_p^-$, say $M_{u_1}$.
Then $u_1$ is a vertex in $V(M_u)\cap V(M_{u_1})$ with $d_{M_{u_1}}(u_1)=p-1$, a contradiction by Lemma~\ref{lem:x in two K}.
\end{proof}
 Let $\mathcal{M}=\{M|M \mbox{ is a } K_{p}^- \mbox{-copy in } G\}$ and $m=|\mathcal{M}|$.
Based on the partition of vertices in $V_2$, we proceed to analyze the following cases.
\setcounter{case}{0}
\begin{case}
$B=\emptyset$.
\end{case}

Since $n\ge p+1$ and $s\le p-2$, $a=n-s-1\ge 2$.

\begin{claim}\label{clm:M cap V1}
Let $M\in \mathcal{M}$. Assume that $x$ and $y$ are two vertices in $M$ with $d_{M}(x)=d_{M}(y)=p-2$. Then $V(M)\cap V_1=\{x, y\}$.
\end{claim}
\begin{proof}
If $V(M)\cap A=\emptyset$, then $V(M)\subseteq N_G[v]$. This is impossible because $|N_G[v]|\le p-1$.
Thus, there exists a vertex $u\in V(M)\cap A$. Since $u\in A$, by Claim~\ref{clm:type A}, $G[N_G[u]]\cong K_p^-$, say $M_1$. Then we claim that $M=M_1$. Otherwise, $u$ is a vertex in $V(M)\cap V(M_1)$ with $d_{M_1}(u)=p-1$, which contradicts Lemma~\ref{lem:x in two K}.
Thus, $M=M_1$ with $x$ and $y$ being two $(p-2)$-degree vertices in $M_1$.

Since $d_{M_1}(u)=p-1$, each vertex in $M_1$ is adjacent to $u$, and thus, $v\notin V(M_1)$.
Together with Claim~\ref{clm:A}, $V(M)\cap V_1=V(M)\backslash (A\cup \{v\})=V(M_1)\backslash (A\cup \{v\})=V(M_1)\backslash A=\{x, y\}$.
\end{proof}
\begin{claim}\label{clm:x in twoM}
Let $M\in \mathcal{M}$ and $x\in V(M)$. If $d_{M}(x)=p-2$, then $x\in V_1$ and $x$ may belong to other $K_p^-$-copies in $\mathcal{M}$; but if $d_{M}(x)=p-1$, then $x\in V_2$, and $x \notin V(M')$ for any other $M' \in \mathcal{M}$.
\end{claim}
\begin{proof}By Claim~\ref{clm:M cap V1}, two $(p-2)$-degree vertices of $M$ belong to $V_1$, while all $(p-1)$-degree vertices of $M$ belong to $V_2$.
By Lemma~\ref{lem:x in two K}, $x$ can be shared only if it has degree $p-2$ in every $M' \in \mathcal{M}$ containing it by two distinct graphs in $\mathcal{M}$. This directly proves this claim.
\end{proof}

By Claims~\ref{clm:M cap V1} and~\ref{clm:x in twoM}, $\sum\limits_{u\in V_1}d_{V_2}(u)\ge 2m(p-2)$.
Therefore, if $2m>s$, then $\sum\limits_{u\in V_1}d_{V_2}(u)\ge(s+1)(p-2)$.
By Claim~\ref{clm:type A} and $B=\emptyset$, we have $\sum\limits_{u\in V_2}d_G(u)=(n-s-1)(p-1)$.
Since $s=d_G(v)\ge 1$,  we have $$2e(G)\ge 2d_G(v)+\sum\limits_{u\in V_1}d_{V_2}(u)+\sum\limits_{u\in V_2}d_G(u)\ge 2s+(s+1)(p-2)+(n-s-1)(p-1)\ge (p-1)n,$$ and $e(G)\ge \frac{(p-1)n}{2}\ge h(n, p)$.
This completes this case.

We now assume that $2m\le s$. Recall our assumption that $N_G(v)=\{v_1,v_2,\ldots, v_s\}$.

\begin{claim}\label{clm:vivj}
Let $i, j\in [s]$ and $i\neq j$. If $\{v_i, v_j\}\not\subseteq V(M)$ for any $M\in \mathcal{M}$, then $v_iv_j\in E(G)$.
\end{claim}
\begin{proof}
Assume that $v_iv_j\in E(\overline{G})$. Then $G+v_iv_j$ has a $K_{p^-}^{+1}$-copy containing $v_iv_j$.
If $v_iv_j$ is the pendant edge of $K_{p^-}^{+1}$, $G[N_G[v_i]]$ or $G[N_G[v_j]]$ contains a $K_p^-$-copy.
Assume without loss of generality that $G[N_G[v_i]]$ contains a $K_p^-$-copy. By Lemma~\ref{lem:N has K_p}, $G[N_G[v_i]]$ is a $K_p^-$-copy in $\mathcal{M}$, say $M'$. Then $v_i\in V(M')\cap V_1$ and $d_{M'}(v_i)=p-1$, a contradiction to Claim~\ref{clm:M cap V1}.
 So, $v_iv_j$ must be an edge in the $K_p^-$-copy of $K_{p^-}^{+1}$.

 We now prove that $G[N_G[v]]\cong K^-_{p-1}$.

  Suppose that $v_i$ and $v_j$ have a common neighbour in $V_2$, say $u$. Since $B=\emptyset$,  we have $V_2=A$ and $u\in A$. Then by Claim~\ref{clm:type A}, $G[N_G[u]]\cong K_p^-$, say $M_1$. Since $\{v_i, v_j\}\subseteq N_G(u)$, $\{v_i, v_j\}\subseteq V(M_1)$, a contradiction to our assumption.
Thus, $N_G(v_i)\cap N_G(v_j)$ is contained in $N_G[v]\backslash \{v_i, v_j\}$.

Since $v_iv_j$ is an edge in a $K_p^-$-copy, by Lemma~\ref{lem:type2}, $|N_G(v_i)\cap N_G(v_j)|\ge p-3$. By our assumption that $d_G(v)\le p-2$, $|N_G[v]\backslash \{v_i, v_j\}|\le p-3$. Together with $\big(N_G(v_i)\cap N_G(v_j)\big)\subseteq \big(N_G[v]\backslash \{v_i, v_j\}\big)$,  we have $N_G(v_i)\cap N_G(v_j)=N_G[v]\backslash \{v_i, v_j\}$ and $|N_G(v_i)\cap N_G(v_j)|=|N_G[v]\backslash \{v_i, v_j\}|=p-3$.
Since $G[N_G(v_i)\cap N_G(v_j)]$ contains a $K_{p-3}$-copy by Lemma~\ref{lem:type2}, $G[N_G(v_i)\cap N_G(v_j)]\cong K_{p-3}$ and $G[N_G(v_i)\cap N_G(v_j)\cup\{v_i, v_j\}]\cong K^-_{p-1}$. As $N_G(v_i)\cap N_G(v_j)\cup\{v_i, v_j\}=N_G[v]$, $G[N_G[v]]\cong K^-_{p-1}$.

Since $n\ge p+1$ and $d_G(v)\le p-2$, $|V_2|=|V(G)\backslash N_G[v]|\ge n-p+1>0$. As $V_2=A$, $A\neq\emptyset$. Take $x\in A$. By Claim~\ref{clm:type A}, $G[N_G[x]]\cong K_p^-$, say $M_2$. By Claim~\ref{clm:M cap V1}, $V(M_2)\cap V_1$ consists of two non-adjacent vertices of $M_2$.
By Observation~\ref{ob:no K_p}, $G$ contains no $K_p$-copy. It follows that these two vertices are also non-adjacent in $G$.
As above, no vertex pair other than $\{v_i, v_j\}$ is non-adjacent in $G[V_1]$. Then $\{v_i, v_j\}\subseteq V(M_2)$.
This contradicts the assumption that $\{v_i, v_j\}\not\subseteq V(M)$ for any $M\in \mathcal{M}$.
\end{proof}

By Claim~\ref{clm:vivj}, $e(G[V_1])\ge \binom{s}{2}-m$.
By Claim~\ref{clm:x in twoM}, for any two $M_1, M_2\in \mathcal{M}$, if $x\in V(M_1)\cap V(M_2)$, then $x\in V_1$.
Thus,
\begin{equation}\label{eq:n}
n=(p-2)m+s+1=p m+(s+1-2m).
\end{equation}
Since $s\le p-2$ and $m\ge 1$, we have $s+1-2m< p-1$.
Together with $2m\le s$, we have $s+1-2m>0$. Thus, $0<s+1-2m<p-1$.
Since $n\equiv k\pmod{p}$, by (\ref{eq:n}), we have $k=s+1-2m$.
Then $h(n, p)=\binom{p}{2}m+\binom{s+1-2m}{2}$.
By Claims~\ref{clm:type A} and~\ref{clm:x in twoM}, $e(V_1, V_2)+e(G[V_2])=\big(\binom{p}{2}-1\big)m$.
Together with our assumption that $2m\le s$, we have
\begin{align*}
e(G) & = d_G(v)+e(G[V_1])+\big(e(V_1, V_2)+e(G[V_2])\big) \\
 & \ge s+\big(\binom{s}{2}-m\big)+\Big(\binom{p}{2}-1\Big)m\\
 & \ge \binom{p}{2}m+\binom{s}{2}\\
  & > \binom{p}{2}m+\binom{s+1-2m}{2}\\
  & =h(n, p).
\end{align*}
\begin{case}
$B\neq\emptyset$.
\end{case}
Recall from Claims~\ref{clm:type A} and~\ref{clm:A} that, for any $u\in A$, $d_G(u)=p-1$, and $d_{\overline{A}}(u)=2$. Thus, $$e(G[A])+e(A, V(G)\backslash A)=\frac{p-3}{2}a+2a=\frac{p+1}{2}a.$$
Since $V_2=A\cup B$ and $A\cap B=\emptyset$, we have
\begin{equation}\label{eq:eg}
\begin{aligned}
e(G) &=e(G[N_G[v]])+e(V_1, A\cup B)+e(G[A\cup B]) \\
& =e(G[N_G[v]])+\Big(e(G[A])+e(A, V(G)\backslash A)\Big)+\Big(e(V_1, B)+e(G[B])\Big) \\
& = e(G[N_G[v]])+\frac{p+1}{2}a+\Big(e(V_1, B)+e(G[B])\Big).
\end{aligned}
\end{equation}

 Recall that $n=tp+k$, where $0\le k\le p-1,\, t\ge 1$ and $n\ge p+1$. Then we have the following claim.
\begin{claim}\label{clm:n lowerbound}
$n\ge p+\frac{5+k^2-pk}{p-4}$.
\end{claim}

\begin{proof}
If $k=0$, then $n\ge p+1$ and $p\ge 5$ imply $t\ge 2$ and $n\ge 2p\ge p+\frac{5}{p-4}$.
 If $k\ge 1$, since $p\ge 5$, $n\ge tp+1\ge tp+\frac{-p+6}{p-4}\ge tp+\frac{(k-1)(k-(p-1))-p+6}{p-4}\ge p+\frac{5+k^2-pk}{p-4}$, as desired.
\end{proof}

\begin{claim}\label{clm:yinB dv1y}
For any $u\in B$, $d_{V_1}(u)\ge p-3$ and $G[V_1\cap N_G(u)]$ contains a $K_{p-3}$-copy.
\end{claim}
\begin{proof}
Since $u\in B$, $G+uv$ has a $K_p^-$-copy containing $uv$. Then by Lemma~\ref{lem:type2},
$G[V_1\cap N_G(u)]$ contains a $K_{p-3}$-copy. Thus, $d_{V_1}(u)\ge p-3$.
\end{proof}

\begin{claim}\label{clm:V1congK p-2}
If $G[V_1]\cong K_{p-2}$, then $e(G)> h(n, p)$.
\end{claim}
\begin{proof}
Assume that $G[V_1]\cong K_{p-2}$.
By Claim~\ref{clm:yinB dv1y}, for any $y\in B$, $d_{V_1}(y)\ge p-3$.

If $d_{V_1}(y)=p-2$, then $G[V_1\cup \{v, y\}]\cong K_p^-$ with $vy\notin E(G)$. Since $G$ is $K_{p^-}^{+1}$-free, $N_G[v_i]=V_1 \cup \{v, y\}$ for each $v_i\in V_1$. Then $y$ is a cut vertex in $G$.
Since $n\ge p+1$, there exists a vertex in $V_2\backslash \{y\}$, say $u$.
Then $d(u, v)=d(u, y)+d(y, v)\ge 3$, and hence $uv$ must be the pendant edge of a $K_{p^-}^{+1}$-copy in $G+uv$. By Claim~\ref{clm:type A}, $G[N_G[u]]\cong K_{p}^-$, say $M_1$.
Let $u_1$ and $u_2$ be two $(p-2)$-degree vertices of $M_1$. Then $d(u_1, v)\ge 3$ or $d(u_2, v)\ge 3$, as $y$ is a cut vertex in $G$. Assume that $d(u_1, v)\ge 3$.
Also, by Claim~\ref{clm:type A}, $G[N_G[u_1]]\cong K_{p}^-$, say $M_2$.
Then $u_1\in V(M_1)\cap V(M_2)$ and $d_{M_2}(u_1)=p-1$, a contradiction by Lemma~\ref{lem:x in two K}.

So, $d_{V_1}(y)= p-3$ for any $y\in B$, and thus, $e(V_1, B)= (p-3)b$.
Note that $b\ge 1$. If $a\ge 1$, then by Claim~\ref{clm:A}, $a\ge p-2$ and $n= p-1+a+b\ge 2p-2$.
Since $p\ge 5$, $p-3\ge \frac{p-1}{2}$. Together with (\ref{eq:eg}), we have
\begin{align*}
e(G) &\ge e(G[N_G[v]])+\frac{p+1}{2}a+e(V_1, B) \\
& = \binom{p-1}{2}+\frac{p+1}{2}a+(p-3)b\\
& \ge  \binom{p-1}{2}+\frac{(p+1)a}{2}+\frac{(p-1)b}{2}\\
& = \frac{(p-1)(p-2+a+b)}{2}+a\\
& =  \frac{(p-1)n}{2}+a-\frac{p-1}{2}\\
& > \frac{(p-1)n}{2}\\
& \ge h(n,p).
\end{align*}
Now, suppose that $a=0$. Since $d_{V_1}(y)= p-3$ for each $y\in B$ and $d_G(y)\ge\delta(G)=p-2$, we have $d_{B}(y)\ge 1$ and $e(G[B])\ge \frac{1}{2}(n-p+1) $. Recall that $n=tp+k\,(0\le k\le p-1, t\ge 1 \mbox{ and } n\ge p+1)$.
By Claim~\ref{clm:n lowerbound}, $n\ge p+\frac{5+k^2-pk}{p-4}> p+\frac{3+k^2-pk}{p-4}$.
Thus, \begin{align*}
e(G) &=e(G[N_G[v]])+e(V_1, B)+e(G[B]) \\
& \ge  \binom{p-1}{2}+(p-3)(n-p+1)+\frac{1}{2}(n-p+1)\\
& =\frac{2p-5}{2}n+\frac{-p^2+4p-3}{2}\\
& =\frac{p-1}{2}n+\frac{p-4}{2}n+\frac{-p^2+4p-3}{2}\\
& > \frac{p-1}{2}n+\frac{k(k-p)}{2}\end{align*}\begin{align*}
& =  h(n,p).
\end{align*}

Summarizing the above, if $G[V_1]\cong K_{p-2}$, then $e(G)> h(n, p)$.
\end{proof}
By Claim~\ref{clm:yinB dv1y}, $G[V_1]$ contains a $K_{p-3}$-copy. Together with $d_G(v)=\delta(G)\le p-2$ and Claim~\ref{clm:V1congK p-2}, it suffices to consider the following cases:
\begin{equation}\label{eq:casesremaining}
G[V_1]\cong K_{p-3}, \mbox{ or } d_G(v)=p-2 \mbox{ and } G[V_1]\ncong K_{p-2},
\end{equation}
and we assume this from now on.

\begin{claim}\label{clm:y in B degree}
For any $y\in B$, $d_{V_1\cup B}(y)\ge p-2$.
\end{claim}

\begin{proof}
By Claim~\ref{clm:yinB dv1y}, $d_{V_1}(y)\ge p-3$.
Assume on the contrary that there exists a vertex $y\in B$ such that $d_{V_1\cup B}(y)=d_{V_1}(y)=p-3$.
Since $y\in B$, $G+vy$ has a $K_{p^-}^{+1}$-copy containing $vy$;
moreover, $vy$ is contained in $M_y$, the unique $K_p^-$-copy of this $K_{p^-}^{+1}$-copy.
Then $d_{M_y}(y)\ge p-2$.

Assume that $d_{M_y}(y)= p-1$. We now analyze the degree of $v$ in $M_y$.
If $d_{M_y}(v)= p-1$, then $vy$ is a $(p-1, p-1)$-edge in $M_y$. Thus, $G[V_1\cap N_G(y)]$ contains a $K_{p-2}^-$-copy, which means $d_{V_1}(y)\ge p-2$, a contradiction.
So, $d_{M_y}(v)=p-2$. Let $N_{M_y}(v)=\{y, v_1, v_2,\ldots, v_{p-3}\}$ and $V(M_y)\backslash N_{M_y}(v)=\{z\}$, see Fig.~\ref{fig:y_in_B}.
 \begin{figure}[H]
\begin{center}
\includegraphics*[width=11cm]{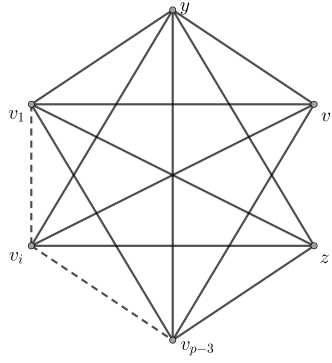}
\end{center}
\vspace*{-0.5cm}
\caption{The $K_p^-$-copy $M_y$ in $G+vy$.}
\label{fig:y_in_B}
\end{figure}

Since $d_{M_y}(y)= p-1$, $yz\in E(G)$. Recall our assumption that $d_{V_1\cup B}(y)=d_{V_1}(y)=p-3$. Together with $\{v_1, v_2,\ldots, v_{p-3}\}\subseteq N_{V_1}(y)$, this gives $z\in A$. Then by Claim~\ref{clm:type A}, $G[N_G[z]]\cong K_{p}^-$, say $M'$. As $v_1, v_2\in N_G(z)$ and $v_1v_2\in E(G)$, $v_1v_2\in E(M')$. Then either $d_{M'}(v_1)=p-1$ or $d_{M'}(v_2)=p-1$.
Assume that $d_{M'}(v_1)=p-1$. Then $M'$ is a $K_p^-$-copy in $G[N_G[v_1]]$. Since $G$ is $K_p$-free by Observation~\ref{ob:no K_p}, $G[N_G[v_1]]\cong K_p^-$, which contradicts Lemma~\ref{lem:type2.2}.

So, $d_{M_y}(y)=p-2$, and thus, $d_{M_y}(v)=p-1$. Then $G[N_{M_y}(v)]$ contains a $K_{p-1}^-$-copy, which implies that $G[V_1]$ contains a $K_{p-2}$-copy, a contradiction to (\ref{eq:casesremaining}).
\end{proof}

\begin{claim}\label{clm:eG}
$e(G) \ge e(G[N_G[v]])+\frac{p+1}{2}a+\frac{2p-5}{2}b$.
\end{claim}

\begin{proof}
By Claim~\ref{clm:yinB dv1y}, $d_{V_1}(y)\ge p-3$ for any $y\in B$.
We can thus divide $B$ into the following two subsets:
$B_1=\{y\in B: d_{V_1}(y)\ge p-2\}$ and $B_2=\{y\in B: d_{V_1}(y)= p-3\}$.
By Claim~\ref{clm:y in B degree}, $d_{B}(y)\ge 1$ for each vertex $y\in B_2$.
Let $b_1=|B_1|$ and $b_2=|B_2|$.
Then $e(V_1, B)+e(G[B])= e(V_1, B_1)+e(V_1, B_2)+e(G[B_1\cup B_2])\ge (p-2)b_1+(p-3)b_2+\frac{1}{2}b_2\ge \frac{2p-5}{2}b$.
By (\ref{eq:eg}), $e(G) \ge e(G[N_G[v]])+\frac{p+1}{2}a+\frac{2p-5}{2}b.$
\end{proof}
\begin{claim}\label{clm:e(G)}
If $e(G)\ge\frac{2p-5}{2}n+ \frac{-p^2+4p-5}{2}$, then $e(G)\ge h(n, p)$.
\end{claim}

\begin{proof}
Note that $n\equiv k\pmod{p}$ and $h(n, p)=\frac{(p-1)n}{2}+\frac{k(k-p)}{2}$.
By Claim~\ref{clm:n lowerbound}, $n\ge p+\frac{5+k^2-pk}{p-4}$.
Then $(p-4)n \ge p^2-4p+5+k^2-pk$, and thus,
$e(G)\ge \frac{2p-5}{2}n+ \frac{-p^2+4p-5}{2}\ge\frac{p-1}{2}n+\frac{k(k-p)}{2}=h(n, p).$
\end{proof}





\begin{claim}\label{clm:twocases}
If $G[V_1]\cong K_{p-3}$, or $d_G(v)=p-2$ and $e(G[N_G[v]])\ge  \frac{p^2-3p}{2}$, then $e(G)\ge h(n, p)$.
\end{claim}

\begin{proof}
We first assume that $2p-5\ge p+1$ and $A\neq \emptyset$. As $a+b=n-s-1$, we have $\frac{p+1}{2}a+\frac{2p-5}{2}b\ge \frac{p+1}{2}(n-s-1)$.
Since $A\neq \emptyset$, by Claim~\ref{clm:A}, $a\ge p-2$. By Claim~\ref{clm:yinB dv1y}, $s=d_G(v)\ge p-3$. As $B\neq \emptyset$, $n= s+1+a+b\ge 2p-3$.
Together with our assumption and Claim~\ref{clm:e(G)},
\begin{align*}
e(G) &\ge e(G[N_G[v]])+\frac{p+1}{2}a+\frac{2p-5}{2}b\\
&\ge e(G[N_G[v]])+\frac{p+1}{2}(n-s-1)\\
&\ge \min\Big\{\binom{p-2}{2}+\frac{p+1}{2}(n-p+2), \frac{p^2-3p}{2}+\frac{p+1}{2}(n-p+1)\Big\}\\
&\ge  \min\Big\{\frac{p+1}{2}n-2p+4, \frac{p+1}{2}n-\frac{3p-1}{2}\}\\
&\ge  \frac{p-1}{2}n\\
&\ge h(n, p).
\end{align*}
Now, we may assume that $p+1\ge 2p-5$ or $a=0$. Then by Claims~\ref{clm:eG} and~\ref{clm:e(G)},
\begin{equation}\label{eq:egeq}
\begin{aligned}
e(G) &\ge e(G[N_G[v]])+\frac{p+1}{2}a+\frac{2p-5}{2}b\\
&\ge e(G[N_G[v]])+\frac{2p-5}{2}(n-s-1)\\
&\ge \min\Big\{\binom{p-2}{2}+\frac{2p-5}{2}(n-p+2), \frac{p^2-3p}{2}+\frac{2p-5}{2}(n-p+1)\Big\}\\
&=\frac{2p-5}{2}n+ \frac{-p^2+4p-5}{2}\\
&\ge h(n, p).
\end{aligned}
\end{equation}

So, if $G[V_1]\cong K_{p-3}$, or $d_G(v)=p-2$ and $e(G[N_G[v]])\ge  \frac{p^2-3p}{2}$, then $e(G)\ge h(n, p)$.
\end{proof}

By (\ref{eq:casesremaining}) and Claim~\ref{clm:twocases}, we only need to consider the case where $d_G(v)=p-2$ and $e(G[N_G[v]])< \frac{p^2-3p}{2}$, so we make this assumption.
Let $V_1=\{v_1, v_2,\ldots, v_{p-2}\}$ and $V'_1=V_1\backslash\{v_{p-2}\}$.
By Claim~\ref{clm:yinB dv1y}, $G[V_1]$ contains a $K_{p-3}$-copy.
So, we suppose that $G[V'_1]\cong K_{p-3}$.  Since $e(G[N_G[v]])< \frac{p^2-3p}{2}$, we have
\begin{equation}\label{eq:dv p-2}
d_{V_1}(v_{p-2})= e(G[N_G[v]])-d_G(v)-e(G[V'_1])<\tfrac{p^2-3p}{2}-(p-2)-\tfrac{(p-3)(p-4)}{2}=p-4.
\end{equation}

\begin{claim}\label{clm:x in B}
For each $x\in B$, $xv_1, xv_2, \ldots, xv_{p-3}\in E(G)$.
\end{claim}

\begin{proof}
Since $x\in B$, $G[V_1\cap N_G(x)]$ contains a $K_{p-3}$-copy by Lemma~\ref{lem:type2}.
Since $|V_1|=p-2$, if $xv_i\notin E(G)$ for some $i\in [p-3]$,
then $v_{p-2}$ forms a $K_{p-3}$-copy with other $p-4$ vertices in $V_1$.
This implies $d_{V_1}(v_{p-2})\ge p-4$, a contradiction to (\ref{eq:dv p-2}).
\end{proof}

\begin{claim}\label{clm:GX clique-}
If $G[N_G[x]]\cong K_p^-$, then $x\in A$.
\end{claim}

\begin{proof}
By (\ref{eq:dv p-2}), $G[N_G(v)]$ contains no $K_{p-2}$-copy. Then by Lemma~\ref{lem:type2.2}, $G[N_G[v_i]]\ncong K_p^-$ for each $i\in [p-2]$. Thus, $x\notin V_1$, and then $x\in V_2$.
If $x\in B$, since $G[N_G[x]]\cong K_p^-$, $vx$ could be the pendant edge of a $K_{p^-}^{+1}$-copy in $G+vx$. Then $x\in A\cap B$, a contradiction to Claim~\ref{clm:AB disjoint}. So, $x\in A$.
\end{proof}
\begin{claim}\label{clm:dx}
If $xy\in E(\overline{G})$ and $G[N_G(x)\cap N_G(y)]$ contains a $K_{p-2}$-copy, then $d_{A}(x)\ge p-2$ and $d_{A}(y)\ge p-2$.
\end{claim}

\begin{proof}
Assume that $M'$ is a $K_{p-2}$-copy in $G[N_G(x)\cap N_G(y)]$. Then for each $z\in V(M')$, $G[N_G[z]]$ contains a $K_p^-$-copy. By Lemma~\ref{lem:N has K_p}, $G[N_G[z]]\cong K_p^-$. Then by Claim~\ref{clm:GX clique-}, $z\in A$. So, $d_{A}(x)\ge p-2$ and $d_{A}(y)\ge p-2$.
\end{proof}

\begin{claim}\label{clm:xv _p-2}
If $xv_{p-2}\in E(\overline{G})$ for some $x\in B$, then $d_{A}(x)\ge p-2$.
\end{claim}

\begin{proof}
Suppose on the contrary that $d_{A}(x)\le p-3$. Then by Claim~\ref{clm:dx}, $G[N_G(x)\cap N_G(v_{p-2})]$ contains no $K_{p-2}$-copy.

Since $xv_{p-2}\in E(\overline{G})$, $G+xv_{p-2}$ contains a $K_{p^-}^{+1}$-copy, say $M$. Then $xv_{p-2}$ cannot be the pendant edge of $M$.
Otherwise, $G[N_G[x]]$ or $G[N_G[v_{p-2}]]$ would contain a $K_p^-$-copy. By Lemma~\ref{lem:N has K_p}, $G[N_G[x]]\cong K_p^-$ or $G[N_G[v_{p-2}]]\cong K_p^-$, which is a contradiction by Claim~\ref{clm:GX clique-}.

Thus, $xv_{p-2}$ must be an edge in the $K_p^-$-copy of $M$. By Lemma~\ref{lem:type2}, $G[N_G(x)\cap N_G(v_{p-2})]$ contains a $K_{p-3}$-copy, say $M_1$. By (\ref{eq:dv p-2}), $d_{V_1}(v_{p-2})\le p-5$.
So, $v\notin V(M_1)$, and thus, $|V(M_1)\cap V_2|\ge 2$. Assume that $y_1, y_2\in V(M_1)\cap V_2$. If $y_j\in A$ for some $j\in [2]$, then $G[N_G[y_j]]\cong K_p^-$ by Claim~\ref{clm:type A}. As $y_j\in N_G(x)\cap N_G(v_{p-2})$ and $xv_{p-2}\notin E(G)$, $x$ and $v_{p-2}$ are $(p-2)$-degree vertices of this $K_p^-$-copy. Then $G[N_G(x)\cap N_G(v_{p-2})]$ contains a $K_{p-2}$-copy, a contradiction.
So, $y_1, y_2\in B$.

Since $y_1y_2$ is an edge in a $K_{p-3}$-copy in $G[N_G(x)]$, we have $G[\{x, y_1, y_2\}]\cong C_3$. Together with $\{x, y_1, y_2\}\subseteq B$ and Claim~\ref{clm:x in B}, $G[\{x, y_1, y_2, v_1, v_2, \ldots, v_{p-3}\}]\cong K_p$, which contradicts Observation~\ref{ob:no K_p}.
\end{proof}

\begin{claim}\label{clm:degree of V2}
For any $u\in V_2$, $d_G(u)\ge p-1$.
\end{claim}

\begin{proof}
By Claim~\ref{clm:type A}, for each $u\in A$, $d_G(u)=p-1$. So, we assume that $u\in B$. Then $uv$ is an edge in a $K_p^-$-copy $M'$ in $G+uv$. So, $uv$ is either a $(p-1, p-1)$-edge, or a $(p-1, p-2)$-edge of $M'$.
If $d_{M'}(v)=p-1$, then $G[N_{M'}(v)]$ contains a $K^-_{p-1}$-copy. Thus, $G[V_1]$ contains a $K^-_{p-2}$-copy, a contradiction to (\ref{eq:dv p-2}). So, $d_{M'}(v)=p-2$ and $uv$ must be a $(p-1, p-2)$-edge of $M'$.
This means that $G[N_G(u)]$ contains a $K_{p-2}$-copy, which implies $d_G(u)\ge p-2$.

Assume that $d_G(u)= p-2$. Then $G[N_G(u)]\cong K_{p-2}$. By Claim~\ref{clm:x in B}, $uv_i\in E(G)$ for each $v_i\in V'_1$. By (\ref{eq:dv p-2}), $d_{V'_1}(v_{p-2})\le p-5$. Together with $G[N_G(u)]\cong K_{p-2}$, $uv_{p-2}\notin E(G)$. Then by Claim~\ref{clm:xv _p-2}, $d_{A}(u)\ge p-2$. So, $d_G(u)\ge d_{V'_1}(u)+d_{A}(u)\ge (p-3)+(p-2)=2p-5>p-2$, a contradiction.

Thus, each vertex in $V_2$ is of degree at least $p-1$ in $G$.
\end{proof}

\begin{claim}\label{clm:v_iv_{p-2}}
Suppose that $d_{V_1}(v_{p-2})\le p-6$ or $d_{B}(v_{p-2})\le 1$. If $v_iv_{p-2}\in E(\overline{G})$ for some $i\in [p-3]$, $G[N_G(v_i)\cap N_G(v_{p-2})]$ contains a $K_{p-2}$-copy.
\end{claim}

\begin{proof}
Suppose on the contrary that $G[N_G(v_i)\cap N_G(v_{p-2})]$ contains no $K_{p-2}$-copy.
Since $v_iv_{p-2}\in E(\overline{G})$, $G+v_iv_{p-2}$ contains a $K_{p^-}^{+1}$-copy $M$.
By the assumption of this case that $G[V_1]\ncong K_{p-2}$ and Lemma~\ref{lem:type2.2}, $G[N_G[v_j]]\ncong K_p^-$ for each $j\in [p-2]$. Then $v_iv_{p-2}$ cannot be the pendant edge of $M$, and thus, $v_iv_{p-2}$ must be an edge in the $K_p^-$-copy of $M$. By Lemma~\ref{lem:type2}, $G[N_G(v_i)\cap N_G(v_{p-2})]$ contains a $K_{p-3}$-copy, say $M_1$. Since $d_{V'_1}(v_{p-2})\le p-5$ by (\ref{eq:dv p-2}), we have $v\notin V(M_1)$, and thus, $|V(M_1)\cap V_2|\ge 2$.

Take $V(M_1)\cap V_2=\{y_1, y_2,\ldots, y_t\}(2\le t\le p-3)$. If $y_j\in A$ for some $j\in [t]$, by Claim~\ref{clm:type A}, $G[N_G[y_j]]\cong K_p^-$. Since $y_j\in N_G(v_i)\cap N_G(v_{p-2})$ and $v_iv_{p-2}\notin E(G)$, $v_i$ and $v_{p-2}$ are two $(p-2)$-degree vertices of this $K_p^-$-copy. Then $G[N_G(v_i)\cap N_G(v_{p-2})]$ contains a $K_{p-2}$-copy, a contradiction to our assumption.
So, $y_j\in B$ for all $j$. This means $d_{B}(v_{p-2})\ge 2$. By our assumption, $d_{V_1}(v_{p-2})\le p-6$.

So, $|V(M_1)\cap V_2|\ge 3$. Since $y_j\in B$ for all $j$, by Claim~\ref{clm:x in B}, $y_jv_i\in E(G)$ for each $v_i\in V'_1$. Note that $G[\{y_1, y_2, y_3\}]$ is a subgraph of $M_1$. Then $G[\{y_1, y_2, y_3\}]\cong K_3$, and thus, $G[V'_1\cup\{y_1, y_2, y_3\}]$ is a $K_p$-copy in $G$, a contradiction by Observation~\ref{ob:no K_p}.
\end{proof}

\begin{claim}\label{clm:dvp-2}
 If $d_G(v)=p-2$ and $d_{V_1}(v_{p-2})\le p-5$, then $e(G)\ge h(n, p)$.
\end{claim}

\begin{proof}
To prove this claim, we consider two cases as follows.
\begin{flushleft}
1.\quad $d_{V_1}(v_{p-2})= p-5$ and $d_{B}(v_{p-2})\ge 2$.
\end{flushleft}

Then $b\ge 2$. Let $G'=G-A$. Then $d_{G'}(v_{p-2})\ge 1+d_{V_1}(v_{p-2})+d_{B}(v_{p-2})\ge p-2$.
Recall our assumption that $G[V'_1]\cong K_{p-3}$. By Claim~\ref{clm:x in B}, $xv_i\in E(G)$ for each $x\in B$ and $v_i\in V'_1$. Then
\begin{align*}
\small\sum\limits_{i=1}^{p-3}d_{G'}(v_i)&\ge d_{V'_1}(v)+\textstyle\sum_{i=1}^{p-3}d_{V'_1\cup B}(v_i)+d_{V_1}(v_{p-2})\\
&=(p-3)+(p-3)(p-4+b)+(p-5)\\
&=(p-3)b+p^2-5p+4,
\end{align*}
and thus,
\begin{equation}\label{eq:sumV1}
\begin{aligned}
\small\sum\limits_{u\in N_G[v]}d_{G'}(u)&=d_{G'}(v)+\small\sum\limits_{i=1}^{p-3}d_{G'}(v_i)+d_{G'}(v_{p-2})\\
&\ge  (p-2)+\big((p-3)b+p^2-5p+4\big)+(p-2)\\
&\ge (p-3)b+p^2-3p.
\end{aligned}
\end{equation}

Assume that $A=\emptyset$. Then $G'=G$. Since $b\ge 2$, by (\ref{eq:sumV1}) and Claim~\ref{clm:degree of V2},
$$
2e(G)=\small\sum\limits_{u\in N_G[v]}d_{G}(u)+\small\sum\limits_{u\in V_2}d_G(u)
\ge   (p-3)b+p^2-3p+(p-1)(n-p+1)
=(p-1)n+(p-3)b-p-1.
$$
If $b\ge 3$, since $p\ge 5$, $e(G)\ge\frac{(p-1)n}{2}\ge h(n,p)$. If $b=2$, then $n=p+1$ and $e(G)\ge \frac{p^2+p-8}{2}> \frac{p^2-p}{2}=h(p+1,p)$.

Now, assume that $A\neq\emptyset$. By Claim~\ref{clm:A}, $a\ge p-2$, and there exist two vertices in $\overline{A}$ each of which has $p-2$ neighbours in $A$, say $u_1$ and $u_2$.
Then $d_A(u_1)\ge p-2$ and $d_A(u_2)\ge p-2$. In addition, $\{u_1, u_2\}\subseteq V_1\cup B$.
If $u_i\in V_1$ for some $i\in[2]$, then by (\ref{eq:sumV1}), $$\small\sum\limits_{u\in N_G[v]}d_{G}(u)\ge \sum\limits_{u\in N_G[v]}d_{G'}(u) +d_A(u_i)\ge (p-3)b+p^2-3p+(p-2)=(p-3)b+p^2-2p-2.$$ Since $b\ge 2$, by Claim~\ref{clm:degree of V2},
\begin{align*}
e(G)&=\frac{1}{2}\big(\sum\limits_{u\in N_G[v]}d_{G}(u)+\sum\limits_{u\in V_2}d_{G}(u)\big)\\
&\ge \frac{1}{2}\big((p-3)b+p^2-2p-2+(p-1)(n-p+1)\big)\\
&\ge\frac{1}{2}\big((p-1)n+2p-9\big)\\
&>\frac{(p-1)n}{2}\\
&\ge h(n,p).
\end{align*}
Now, assume that $\{u_1, u_2\}\subseteq B$. By Claim~\ref{clm:x in B}, $u_iv_j\in E(G)$ for each $i\in [2]$ and $v_j\in V'_1$. Then $d_G(u_i)\ge d_A(u_i)+d_{V'_1}(u_i)\ge p-2+p-3=2p-5$.
By Claim~\ref{clm:degree of V2}, $\sum\limits_{u\in V_2}d_{G}(u)=d_G(u_1)+d_G(u_2)+\sum\limits_{u\in V_2\backslash\{u_1, u_2\}}d_{G}(u)\ge 2(2p-5)+(p-1)(n-p-1)=(p-1)n-p^2+4p-9$.
By (\ref{eq:sumV1}) and $b\ge 2$,
\begin{align*}
\sum\limits_{u\in N_G[v]}d_{G}(u)+\sum\limits_{u\in V_2}d_{G}(u)
&\ge \Big((p-3)b+p^2-3p\Big)+\Big((p-1)n-p^2+4p-9\Big)\\
&\ge(p-1)n+3p-15\\
&\ge (p-1)n.
\end{align*}
Then $e(G)=\frac{1}{2}\big(\sum\limits_{u\in N_G[v]}d_{G}(u)+\sum\limits_{u\in V_2}d_{G}(u)\big)\ge \frac{(p-1)n}{2}\ge h(n, p)$.

Thus, if $d_{V_1}(v_{p-2})= p-5$ and $d_{B}(v_{p-2})\ge 2$, then $e(G)\ge h(n, p)$.
\begin{flushleft}
2.\quad $d_{V_1}(v_{p-2})\le p-6$ or $d_{B}(v_{p-2})\le 1$
\end{flushleft}

 By Claim~\ref{clm:x in B}, $xv_i\in E(G)$ for each $x\in B$ and $v_i\in V'_1$.
By Claims~\ref{clm:dx} and~\ref{clm:v_iv_{p-2}}, for each $v_iv_{p-2}\in E(\overline{G})$, $d_{A}(v_i)\ge p-2$. Since $G[V'_1]\cong K_{p-3}$,
 $$d_G(v_i)=1+d_{V_1}(v_i)+d_{A}(v_i)+d_{B}(v_i)\ge 1+(p-4)+(p-2)+b=2p-5+b.$$
For $v_jv_{p-2}\in E(G)$, $$d_G(v_j)\ge 1+d_{V_1}(v_j)+d_{B}(v_j)\ge 1+(p-3)+b=p-2+b.$$
Since $d_G(v_{p-2})\le p-5$, we can assume that $v_1v_{p-2}, v_2v_{p-2}\notin E(G)$.
Recall that $\mathcal{M}=\{M|M \mbox{ is a } K_{p}^- \mbox{-copy in } G\}$.
By Claim~\ref{clm:v_iv_{p-2}}, $\{v_1, v_{p-2}\}\subseteq V(M_1)$ and $\{v_2, v_{p-2}\}\subseteq V(M_2)$ for some $M_1, M_2\in \mathcal{M}$. By Lemma~\ref{lem:x in two K}, the $(p-1)$-degree vertices of any two distinct $K_p^-$-copies are distinct. Then $d_{M_1}(v_{p-2})+d_{M_2}(v_{p-2})=2p-4$. Also, by Claim~\ref{clm:GX clique-}, each $(p-1)$-degree vertex of $M\in \mathcal{M}$ belongs to $A$. Then $d_A(v_{p-2})\ge d_{M_1}(v_{p-2})+d_{M_2}(v_{p-2})=2p-4$. By our assumption that $b\ge 1$, we have
\begin{align*}
\sum\limits_{u\in N_G[v]}d_G(u) &=d_G(v)+\sum\limits_{u\in V'_1}d_G(u)+d_G(v_{p-2})\\
&\ge (p-2)+\Big((p-5)(p-2+b)+2(2p-5+b)\Big)+(2p-3)\\
&= p^2+(p-3)b-5\\
&\ge p^2+p-8\\
&> p^2-2p+1.
\end{align*}
Along with Claim~\ref{clm:degree of V2}, we have
$$e(G)= \frac{1}{2}\big(\sum\limits_{u\in N_G[v]}d_G(u)+\sum\limits_{u\in V_2}d_G(u)\big)
> \frac{1}{2}\big(p^2-2p+1+(n-p+1)(p-1)\big)=\frac{(p-1)n}{2}\ge h(n, p),$$ as desired.
\end{proof}
By Claim~\ref{clm:yinB dv1y}, $G[V_1]$ contains a $K_{p-3}$-copy.
Together with Claim~\ref{clm:V1congK p-2}, Claim~\ref{clm:twocases}, and Claim~\ref{clm:dvp-2}, our argument holds in Case 2.

This completes the proof.
\end{proof}
Now we are in a position to prove Theorem~\ref{thm:Kp*}.
\begin{proof}[\textbf{Proof of Theorem~\ref{thm:Kp*}}]
Since $n\equiv k\pmod{p}$, the graph $\frac{n-k}{p}K_p\cup K_k$ is a $K_{p^-}^{+1}$-saturated graph of order $n$ with $h(n, p)$ edges. So, $\sat(n, K_{p^-}^{+1})\le e(\frac{n-k}{p}K_p\cup K_k)=h(n, p)$.
Now we prove the lower bound.

Let $G$ be a minimum $K_{p^-}^{+1}$-saturated graph of order $n$. If $G$ is connected, then by Lemma~\ref{lem:connected}, $e(G)\ge h(n,p)$. Now suppose that $G$ is disconnected, and $G_{1}, G_{2},\ldots, G_{s}\,(s\geq 2)$ are the components of $G$.

By Lemma~\ref{kpsubgraph}, we may assume that $G_i\cong K_p$ for each $i\in [s-1]$.
For the remaining component $G_s$ of order $n_s$, if $n_s\le p$, since $n\equiv k\pmod{p}$, we have $n_s=k$. Then $G_s \cong K_k$, for otherwise, adding edges in $G_s$ produces no copy of $K_{p^-}^{+1}$, a contradiction. So, $e(G)=\frac{(p-1)(n-k)}{2}+\binom{k}{2}=h(n, p)$.
If $n_s\ge p+1$, then $G_s$ is a connected $K_{p^-}^{+1}$-saturated graph. Since $G_i\cong K_p$ for each $i\in [s-1]$, we have $n_s\equiv k\pmod{p}$. By Lemma~\ref{lem:connected}, $e(G)\ge\frac{(p-1)(n-n_s)}{2}+h(n_s,p)= \frac{(p-1)(n-n_s)}{2}+\frac{(p-1)(n_s-k)}{2}+\binom{k}{2}=h(n, p)$.

Summarizing the above, if $p\ge 5$, we have $e(G)\ge h(n,p)$.
Thus, $\sat(n, K_{p^-}^{+1})=h(n, p)$.
\end{proof}

\section{Concluding remark}

This work builds on the inequality proved by Cameron and Puleo \cite{Ca}, which states that $\sat(n,K_1 \vee F)\le n-1+\sat(n-1, F)$ for all $n \ge |V(F)|+1$. This bound immediately raises the question of when equality holds, and much of the subsequent work has focused on a non-empty graph $F$ with no isolated vertex.
In this paper, we investigate the case in which $F$ is a non-empty graph with isolated vertices.
We determine the exact saturation numbers for $K_1\vee F$ when $F=K^-_{3}\cup sK_1(s\ge 1)$ and $F=K^-_{p-1}\cup K_1(p\ge 5)$ and verify whether
\begin{equation}\label{eq:ca}
\sat(n,K_1 \vee F)= n-1+\sat(n-1, F)
\end{equation}
holds for these cases.

In fact, both graphs $F$ we investigate here are the union of a nearly complete graph $K_{p-1}^-$ and some isolated vertices.
Interestingly, despite the very similar structures of these two families of graphs, we obtain two completely opposite conclusions. For $F=K^-_{3}\cup sK_1(s\ge 1)$, it can be easily checked that $\sat(n-1, K_{3}^-\cup sK_1)=\lfloor\frac{n-1}{2}\rfloor$ for any $n\ge s+4$. Together with Theorem~\ref{thm:K4*}, $$\sat(n, K_1\vee (K_{3}^-\cup sK_1))=\Big\lceil\frac{3n-4}{2}\Big\rceil=n-1+\Big\lfloor\frac{n-1}{2}\Big\rfloor=n-1+\sat(n-1, K_{3}^-\cup sK_1).$$
Moreover, as shown in our proof, there exists an extremal graph of $K_1\vee (K_{3}^-\cup sK_1)$ that contains a full-degree vertex, just like the construction in Eq.~(\ref{eq:ca}).

For $F=K^-_{p-1}\cup K_1(p\ge 5)$, however, the equality (\ref{eq:ca}) does not hold.
From the saturation number of generalized books given in~\cite{GChen2009}, we immediately obtain that if $n\ge 4(2p-4)^{p-3}$, then $\sat(n-1, K_{p-1}^-\cup K_1)=\frac{1}{2}\big((2p-7)n-p^2+4p-2+\theta\big)$, where $\theta=0$ if $n-p\equiv 0 \pmod{2}$, and $\theta=1$ otherwise. It can be checked that $\sat(n, K_1\vee (K_{p-1}^-\cup K_1))\neq n-1+\sat(n-1, K_{p-1}^-\cup K_1)$. Moreover, we claim that any extremal graph for $K_1\vee (K_{p-1}^-\cup K_1)$ contains no full-degree vertex; this is proved in the following theorem.

\begin{theorem}
Let $n\ge p+1$.
If $p\ge 5$, then any extremal graph of $K_1\vee (K_{p-1}^-\cup K_1)$ contains no full-degree vertex.
\end{theorem}

\begin{proof}
For convenience, we proceed directly using the same notation, definitions, constructions, and intermediate results established in the proof of Lemma~\ref{lem:connected}.

Assume that $G$ is an extremal graph of $K_{p^-}^{+1}$ and $G$ contains a full-degree vertex $u$.
Recall that $v$ is a vertex of minimum degree in $G$ with $V_1=N_G(v)$ and $V_2=V(G)\backslash (V_1\cup \{v\})$.
Clearly, $G$ is connected and $u\in V_1$.
Since $e(G)=h(n, p)$, we have $\delta(G)\le p-2$, for otherwise, $$e(G)= \frac{1}{2}\sum\limits_{x\in V(G)}d_G(x)\ge \frac{1}{2}\big(n-1+(p-1)(n-1)\big)=\frac{pn-p}{2}>\frac{(p-1)n}{2}\ge h(n, p),$$ a contradiction.


Take $w\in V_2$. Then $G+vw$ contains a $K_{p^-}^{+1}$-copy $M_w$ with $vw\in E(M_w)$.
Recall that $V_2$ is partitioned into two disjoint subsets $A$ and $B$,
where, for $w \in A$, $vw$ is the pendant edge in $M_w$, and for $w \in B$, $vw$ is an edge in the $K_p^-$-copy of $M_w$.
Also, $A\cap B=\emptyset$ by Claim~\ref{clm:AB disjoint}.
Let $a=|A|$ and $b=|B|$. Then $a+b=n-1-d_G(v)$.

If $A\neq\emptyset$, then we can assume that $w\in A$. By Claim~\ref{clm:type A}, $G[N_G[w]]\cong K_p^-$, say $M'$. Since $d_G(u)=n-1$, $uw\in E(M')$. If $d_{M'}(u)=p-1$, then $M'$ is a $K_p^-$-copy in $G[N_G[u]]$.
By Lemma~\ref{lem:N has K_p}, $G[N_G[u]]\cong K_p^-$ and $d_G(u)=p-1$. This contradicts the fact that $d_G(u)=n-1$.
So, $d_{M'}(u)=p-2$.
Let $u'$ be a vertex in $V(M')\backslash\{u\}$ with $d_{M'}(u')=p-2$.
Since $G$ is $K_p$-free, $uu'\notin E(G)$. This contradicts $d_G(u)=n-1$.

So, $A=\emptyset$, which implies $b= n-1-d_G(v)\ge n-1-(p-2)\ge 2$. By Claim~\ref{clm:yinB dv1y}, $G[V_1]$ contains a $K_{p-3}$-copy.
From the proof of Case 2 in Lemma~\ref{lem:connected},
we consider the following cases.
\begin{enumerate}
  \item $G[V_1]\cong K_{p-2}$. By Claim~\ref{clm:V1congK p-2}, we have $e(G)>h(n, p)$, a contradiction.
  \item $G[V_1]\cong K_{p-3}$, or $\delta(G)=p-2$ and $\frac{p^2-3p}{2}\le  e(G[N_G[v]])<\frac{p^2-3p+2}{2}$.

  Recall that $A=\emptyset$. By our proof of Claim~\ref{clm:twocases}, if $e(G)=h(n, p)$, then
  all intermediate steps in (\ref{eq:egeq}) hold with exact equality.
  Thus, $e(G) = e(G[N_G[v]])+\frac{p+1}{2}a+\frac{2p-5}{2}b$, $\delta(G)=p-2$, $e(G[N_G[v]])=\frac{p^2-3p}{2}$ and $$e(G)=\frac{2p-5}{2}n+ \frac{-p^2+4p-5}{2}=\frac{(p-1)n}{2}+\frac{k(k-p)}{2}=h(n, p).$$ Thus, $n= p+\frac{5+k^2-pk}{p-4}$. Then by the proof of Claim~\ref{clm:n lowerbound}, we have $p=5$.

  Since $d_G(v)=\delta(G)=p-2$ and $e(G[N_G[v]])=\frac{p^2-3p}{2}$, we have $G[N_G[v]]\cong K_{p-1}^-$.
  Together with $p=5$, we have $G[N_G[v]]\cong K_4^-$.
  Let $V_1 \backslash\{u\}=\{v_1, v_2 \}$. As $d_G(u)=n-1$, $v_1v_2$ must be the one edge not in $G[N_G[v]]$. Then $G+v_1v_2$ has a $K_{p^-}^{+1}$-copy containing $v_1v_2$, say $M$.
  Since $G[V_1]\ncong K_{p-2}$, by Lemma~\ref{lem:type2.2}, $G[N_G[v_i]]\ncong K_p^-$ for each $i\in [2]$. Thus, $v_1v_2$ must be an edge in the $K_p^-$-copy of $M$.

  Recall we have shown that $e(G)=e(G[N_G[v]])+\frac{p+1}{2}a+\frac{2p-5}{2}b$ and $A=\emptyset$.
  By our proof of Claim~\ref{clm:eG}, one can deduce that $d_G(y)=p-2$ and $d_{V_1}(y)=p-3$ for any $y\in B$.
  Thus, $u, v_1$, and $v_2$ are all vertices in $G+v_1v_2$ of degree at least $p-1$. So, all three vertices are in $M$ and have degree at least $p-1$ in $M$. Furthermore, by the construction of a $K_p^-$-copy, two $(p-2)$-degree vertices in $M$ must be adjacent to $u, v_1$ and $v_2$. Since $V(G)\backslash \{u, v_1, v_2\}=B\cup \{v\}$, there exists a vertex in $B$, say $y'$, such that $d_{V_1}(y')=3=p-2$. This contradicts the property $d_{V_1}(y')=p-3$ proved above.

  \item $d_G(v)=p-2$ and $e(G[N_G[v]])<\frac{p^2-3p}{2}$. Recall our assumption that $G[\{v_1, \ldots, v_{p-3}\}]\cong K_{p-3}$. Then $d_{V_1}(v_{p-2})\le p-5$.
      Since $e(G)=h(n, p)$ and $A=\emptyset$,  by our proof in Claim~\ref{clm:dvp-2}, we have $d_{V_1}(v_{p-2})= p-5$ and $p=5$.
   This implies that $v_{p-2}$ has no neighbour in $V_1$. This contradicts the fact that $u\in V_1$ and $d_G(u)=n-1$.
\end{enumerate}

Summarizing the above, any extremal graph of $K_1\vee (K_{p-1}^-\cup K_1)$ contains no full-degree vertex.
\end{proof}

From the above discussion, we reasonably conjecture that the presence of a full-degree vertex in some extremal graph of $K_1\vee F$ is a necessary condition for Eq.~(\ref{eq:ca}) to hold. Our precise conjecture is as follows.
\begin{conjecture}
Let $F$ be a non-empty graph and $n \ge |V(F)|+1$.
Then $\sat(n, K_1 \vee F)=n-1+\sat(n-1, F)$ holds only if there exists an extremal graph of $K_1 \vee F$ containing a full-degree vertex.
\end{conjecture}
Furthermore, it would also be interesting to investigate the conditions under which an extremal graph has a full-degree vertex.
While this problem is perhaps too general to be meaningful for arbitrary graphs, the situation becomes entirely different when considered for joins of graphs.

Turning to the results in our paper, it is natural to pose the following problem.
\begin{problem}
Determine the exact value of $\sat(n, K_{p^-}^{+s})$ for all $p\ge 5$ and $s\ge 2$.
\end{problem}

The \emph{house graph} is a five-vertex graph formed by adding a chord to a 5-cycle.
In this paper, we determine the saturation number for $K_{4^-}^{+1}$. Together with previous results, this leaves the house graph as the only connected graph on 5 vertices whose saturation number is still unknown.
We pose the following problem.

\begin{problem}
Determine the saturation number for the house graph.
\end{problem}

\section*{Acknowledgments}
This research is supported in part by National Natural Science Foundation of China (No. 12571363) and   Natural Science Foundation of Hunan Province (Grant No. 2025JJ30003).


\end{document}